\documentclass[preprint,12pt]{elsarticle}
\usepackage{graphicx}
\usepackage{amsmath,amssymb,latexsym}
\usepackage[colorlinks, plainpages]{hyperref}
\usepackage[english,francais]{babel}
\usepackage{amsmath,amssymb,latexsym}

\newtheorem{theorem}{Theorem}[section]

\newenvironment{proof}{\par\noindent\textbf{Proof}.}{\hfill$\square$}
\newtheorem{proposition}[theorem]{Proposition}

\newtheorem{definition}[theorem]{definition}

\newcommand{\be}{\begin{equation}}
\newcommand{\ee}{\end{equation}}
\newcommand{\beq}{\begin{eqnarray}}
\newcommand{\eeq}{\end{eqnarray}}
\def\og{\leavevmode\raise.3ex\hbox{$\scriptscriptstyle\langle\!\langle$~}}
\def\fg{\leavevmode\raise.3ex\hbox{~$\!\scriptscriptstyle\,\rangle\!\rangle$}}

\begin{document}

\begin{frontmatter}

\selectlanguage{english}
\title{Affine Metrics and Associated Algebroid Structures: Application to General Relativity}

\selectlanguage{english}
\author{Gh. Fasihi Ramandi}
\ead{gh$_{-}$fasihi@aut.ac.ir}
\begin{abstract}
In this paper, affine metrics and its associated algebroid bundle provide an structure for unification of gravity and electromagnetism and, geometrization of matter. The field equations are derived from an action principal. The derived equations contain Einstein and  Maxwell equations in vacuum simultaneously. 
\vskip 0.5\baselineskip
\selectlanguage{francais}
\end{abstract}
\end{frontmatter}


\selectlanguage{english}
\section{Introduction}
\label{intro}
Attempts at the unification of gravitation and electromagnetism have been made ever since the advent of general relativity. Most of these attempts share the idea that Einstein's original theory must in some way be generalized that some part of geometry describes electromagnetism.
\\
Early attempts in unification of fields have been done by Weyl (considering conformal structures and introducing a new gauge transformation, 1918), Kaluza (adding additional dimension to space-time, 1919), Eddington (considering connection as the central concept and decomposing its Ricci tensor to symmetric and anti-symmetric parts, 1921), Schouten (considering connections with nonzero torsion, 1921), Klein (interpreting fifth dimension of Kaluza theory as a relation to quantum concepts, 1926), Infeld (considering asymmetric metric that its symmetric part represent gravity and its anti-symmetric part represent electromagnetism
field, 1928), Einstein and Mayer (5-vector formalism and considering vector bundles and connections on vector bundles, 1931) \cite{history}.
\\
The due problem has been investigated through many others’ experiments so far. In \cite{affine 2} affine metrics and its associated algebroid bundle are developed and, theses structures are applied to the general relativity to provide an structure for unification of gravity and electromagnetism. But in \cite{affine 2}, authores have used an special affine metric and constructing the field equations is based on a good suggestion. Here, we consider more general case and obtain field equation from a principal action and calculus of variations. The derived equations not only contain Maxwell ane Einstein equations, but also describe the effects of so-called dark matter in field equations.
\section{Preliminaries}
In this section, we summarize definitions and essential facts about affine metric structures, with emphasis on those aspects that are most important for this paper. For more details, you can consult (\cite{affine 1},\cite{affine 2}).
\subsection{Affine Concepts}
\begin{definition}
Suppose that V and W are vector spaces, then a function
 $S:V\longrightarrow W$ is called an affine function if there
  exists a linear function $T:V\longrightarrow W$ such that for all $u,v\in V$ we have:
\begin{align*}
 S(u+v)=S(u)+T(v)
\end{align*}
\end{definition}
$T$ is unique and is known as the linear part of S. Every affine function $S:V\longrightarrow W$ is in the form of $S(u)=T(u)+a$ such that $T$ is linear and $a\in W$.\\
In this paper $V$ is a finite dimensional vector space.
\begin{definition}
If a function $S:V\times V\longrightarrow I\!\!R$ is an affine function for each variable separately, then it will be a 2-affine function. And if it is affine for the first variable and linear in the second variable, then it is called an affine-linear function. Similary, S is called linear-affine if it is linear for the first variable and affine for the second variable.
\end{definition}
\begin{definition}
If $S:V\times V\longrightarrow I\!\!R$ is a 2-affine function such that for all $a,b\in E, S(a,b)=S(b,a)$, then S is called symmetric.
\end{definition}
If $S:V\times V\longrightarrow I\!\!R$ is a symmetric 2-affine function, then there exists a unique linear-affine function $T_{1}$ and a unique symmetric bilinear function $T$ such that for all $a,b,u,v\in V$ we have:
\begin{align*}
S(a+u,b+v)=S(a,b)+T_{1}(u,b)+T_{1}(v,a)+T(u,v)
\end{align*}
$T_{1}$ is the linear-affine part of S and T is the bilinear part of S.
\begin{definition}
If a 2-affine function $S:V\times V\longrightarrow I\!\!R$ is symmetric and its bilinear part is an inner product on V, then $S$ is called an affine inner product on V.
\end{definition}
Each inner product on a vector space V as affine space, is also an affine inner product.\\
\textbf{Notation}
Let $a,b,u,v\in V$. An affine inner product on V is shown by $(a,b)$ and its affine-linear, linear-affine and biliniear parts respectively are shown by $(a,v\rangle$, $\langle u,b)$ and $\langle u,v\rangle$. Hence
\begin{align*}
(a+u,b+v)=(a,b)+(a,v\rangle +\langle u,b)+\langle u,v\rangle .
\end{align*}
Suppose $V$ is a vector space and $(.,.)$ is an affine inner product on $V$, then there exist a unique vector $z\in V$ and a unique scalar $\lambda$ such that
\begin{align*}
(u,v)=\lambda +\langle u-z,v-z\rangle
\end{align*}
For each inner product $\langle .,.\rangle$ on $V$ and vector $z\in V$ and scalar $\lambda$, the above equation defines an affine inner product on $V$.\\
In the first case it is sufficient to assume $z$ as the vector that for all $v\in V$, $\langle z,v\rangle=-(0,v\rangle$ and $\lambda =(0,0)-\langle z,z\rangle$. All affine inner products on $V$ are obtained in this way.
\subsection{associated inner product space to affine metrics}
Let $(.,.)$ be an affine inner product on $V$ . Set $\hat{V}$ be the space of real valued affine map on $V$ . $\hat{V}$ is a vector space whose dimension is one plus dimension of $V$ . For all $x\in V$ set $\hat{x}: V\to I\!\!R$ be the affine map $\hat{x}(y) = (x, y)$. The map $x\mapsto \hat{x}$ is affine and imbed $V$ into $\hat{V}$ as an affine subspace. For all $x\in V$ set $\bar{x} : V \to I\!\!R$ be the affine map $\bar{x}(y) =< x, y)$. The map $x\mapsto \bar{x}$ is linear and imbed $V$ into $\hat{V}$ as a vector subspace. Denote the set of all $\bar{x}$ by $\bar{V}$ . The space of real valued constant function on $V$ is a one dimensional subspace of $\hat{V}$ and is complementary to $\bar{V}$ . So, we find a natural projection $\rho: \hat{V}\to V$ whose kernel is constant functions and its restriction to $\bar{V}$ is $\bar{x}\mapsto x$.\\
If $(x, y) = \lambda+ \langle x-z, y-z \rangle$ and $\lambda\neq 0$, then there exist a unique inner product on $\hat{V}$ such that for all $x, y \in V$ we have $\langle \hat{x},\hat{y}\rangle = (x, y)$.\\
By manipulating this property we can find the right definition of this inner product. $\hat{z}$ is the constant function $\hat{z} (x) =\lambda$ and must be orthogonal to $\bar{V}$ . Every element of $\hat{V}$ is uniquely written in the form $\bar{x}+\mu \hat{z}$, and we must define
\begin{equation*}
\langle \bar{x}+\mu_1 \hat{z},\bar{y}+\mu_2 \hat{z} \rangle =\langle x, y \rangle +\lambda \mu_1 \mu_2.
\end{equation*}
Note that for all $x, y \in V$ we have $\hat{x}+\bar{y}=\widehat{x+y}$, so $\hat{x}= \overline{x-z} +\hat{z}$.
\section{Affine semi-Riemannian manifolds and geometry of its associated algebroid}
In this section, $M$ is a fixed smooth manifold and all functions are smooth.\\
\begin{definition}
If for every $p\in M$, we choose on every $T_p M$ an affine inner product smoothly, then we call it an affine metric on $M$ and $M$ is called an affine semi-Riemannian manifold.
\end{definition}
Every semi-Riemannian manifold is also an affine semi-Riemannian manifold. The bilinear part of an affine metric on M is a regular semi-Riemannian metric on $M$ and it is called the associated semi-Riemannian metric.\\
\textbf{Example} 
If $\langle .,.\rangle$ be a semi-Riemannian metric on $M$, and $A\in \mathcal{X}(M)$ and $\phi\in C^\infty (M)$, then the following formula defines an affine metric on M.
\begin{equation*}
\forall X, Y \in \mathcal{X}(M)\qquad (X, Y ) =\phi + \langle X- A, Y- A \rangle.
\end{equation*}
Every affine metric on $M$ can be written as above. If 0 be the zero vector field, it is sufficient to set $A$ be the vector field which for all $X\in \mathcal{X}(M)$, $\langle A,X \rangle = −(0,X \rangle$ in which $\langle ., .\rangle$ is the bilinear part of the affine metric and set $\phi = (0, 0)-\langle A,A\rangle$.
\subsection{Associated algebroid bundle and its geometry}
In this section $M$ is an affine semi-Riemannian manifold and $X, Y \in \mathcal{X}(M)$ and its affine metric is as follows:
\begin{equation*}
(X, Y ) = e^{2\theta} + \langle X- A, Y- A \rangle.
\end{equation*}
Let $\widehat{TM}$ be the vector bundle $\cup_{p\in M} \widehat{T_p M}$. For every $X\in \mathcal{X}(M)$ let $\hat{X}$ and $\overline{X}$ be sections of $\widehat{TM}$ such that $(\hat{X})_p = \hat{X}_p , ( \overline{X} )_p =\overline{X_p}$. $\hat{A}$ is the constant function $\hat{A}(u)=1$ and for simplicity we denote it by $G$. Let $\overline{TM}$ be the vector bundle 
$\cup_{p\in M} \overline{T_p M}$ that is a subvector bundle of $\widehat{TM}$. $\overline{TM}$ is complementary to line
subbundle generated by $G$. Every section of $\widehat{TM}$ uniquely written in the form $\overline{X}+ fG$ for some $X\in \mathcal{X}(M)$ and $f\in C^\infty (M)$. $\widehat{TM}$ is a semi-Riemannian vector bundle by the induced inner product:
\begin{equation*}
X, Y \in \mathcal{X}(M) , f, g \in C^\infty (M)\qquad \langle \overline{X}+ fG, \overline{Y}+ gG\rangle =\langle X, Y \rangle +fge^{2\theta}
\end{equation*}
$\widehat{TM}$ has a natural algebroid structure over $TM$. The anchor map is
\begin{equation*}
\begin{array}{c}
\rho:\widehat{TM}\longrightarrow TM \\ 
\overline{X}+fG\longmapsto X
\end{array} 
\end{equation*}
In the definition of Lie bracket on $\widehat{TM}$ the vector field $A$ make a crucial role. Lie bracket on $\widehat{TM}$ is defined as follows:
\begin{equation*}
[\overline{X},\overline{Y}]=\overline{[X,Y]}+2\Omega (X,Y)G,\qquad [\overline{X},G]=0
\end{equation*}
where,
\[
\Omega (X,Y)= \dfrac{1}{2}(\langle \nabla_X A , Y\rangle -\langle \nabla_Y A, X\rangle )
\]
Jacobi identity is hold because $\Omega=\dfrac{1}{2}dA^\flat$ is a closed form. Also, $\Omega$ satisfies in the following differential identity. For all $X,Y \text{and}\, Z\in \mathcal{X}(M)$,
\begin{equation}\label{bianchi}
(\nabla_X \Omega)(Y,Z)+(\nabla_Y \Omega)(Z,X)+(\nabla_Z \Omega)(X,Y)=0.
\end{equation}
Now, $\widehat{TM}$ is a semi-Riemannian algebroid over $TM$ and has a unique Levi-civita connection $\hat{\nabla}$ which can be computed by the following relation \cite{buc}. For all $X,Y \text{and}\, Z\in \mathcal{X}(M)$,
\begin{align*}
2\langle \widehat{\nabla}_{X}Y,Z\rangle &= \rho(X)\langle Y,Z\rangle +\rho(Y)\langle X,Z\rangle
 -\rho(Z)\langle X,Y\rangle\\ &+\langle [X,Y],Z\rangle -\langle [Y,Z],X\rangle +\langle [Z,X],Y\rangle
\end{align*}

Put $\Omega_\theta (X,Y)=e^\theta \Omega (X,Y)$, in foregoing computations we need to use some other tensors equivalent to $\Omega_\theta$. For $X,Y \in \mathcal{X}(M)$, define anti-symmetric 1-1 tensor field $F_\theta$ on $M$ as follows
\begin{equation*}
\langle F_\theta (X),Y\rangle =\Omega_\theta (X,Y).
\end{equation*}
With respect to the metric on $\widehat{TM}$, the secrion $G$ is not normal, $<G,G>=e^{2\theta}$. Set $G_1=e^{-\theta}G$ which is obviously normal field, $\langle G_1 ,G_1 \rangle=1$.
In terms of structures defined above, for all $X\in\mathcal{X}(M)$, we have
\[
[\overline{X},\overline{Y}]=\overline{[X,Y]}+2\Omega_\theta(X,Y)G_1 ,\qquad [\overline{X},G_1 ]=-\langle \vec{\nabla}\theta , X \rangle G_1 
\]

The gradient of a smooth function $f$ on a semi-Riemannain manifold $(M,g)$ is denoted by $\nabla f$ and defined by $g(\nabla f,X)=df(X)$.\\
\begin{proposition}
Levi-Civita connection of the semi-Riemannian algebroid $\widehat{TM}$ satisfies the following relations.
\begin{align}
\hat{\nabla}_{G_1 } G_1  &=-\overline{\nabla\theta}\label{a}\\
\hat{\nabla}_{\overline{X}} G_1  &=-\overline{F_\theta (X)}\\
\hat{\nabla}_{G_1 } \overline{X} &=-\overline{F_\theta(X)}+d\theta(X)G_1 \\
\hat{\nabla}_{\overline{X}} \overline{Y}&=\overline{\nabla_X Y}+\Omega_\theta (X,Y)G_1 
\end{align}
\end{proposition}
\begin{proof}
Straightforward computations show these results. For example we verify (\ref{a}).
\begin{align*}
2\langle \hat{\nabla}_{G_1 } G_1  ,G_1  \rangle &=\rho(G_1 )\langle G_1 ,G_1 \rangle +\rho(G_1 )\langle G_1 ,G_1 \rangle -\rho(G_1 )\langle G_1 ,G_1 \rangle\\
&+\langle [G_1 ,G_1 ],G_1 \rangle -\langle [G_1 ,G_1 ],G_1 \rangle +\langle [G_1 ,G_1 ],G_1 \rangle =0
\end{align*}
Also, 
\begin{align*}
2\langle \hat{\nabla}_{G_1 } G_1  ,\overline{X} \rangle &=\rho(G_1 )\langle G_1 ,\overline{X}\rangle +\rho(G_1 )\langle \overline{X} ,G_1 \rangle -\rho(\overline{X})\langle G_1 ,G_1 \rangle\\
&+\langle [G_1 ,G_1 ],\overline{X}\rangle-\langle [G_1 ,\overline{X}],G_1 \rangle+\langle [\overline{X},G_1 ],G_1 \rangle\\
&=2\langle [\overline{X},G_1 ],G_1 \rangle=-2\langle \vec{\nabla}\theta ,X \rangle
\end{align*}
\end{proof}

For the sake of simplicity in the sequel, let us agree $S_\theta$ denotes the following symmetric tensor.
\[
S_\theta=Hes(\theta)+d\theta \otimes d\theta
 \]
Note that, for a smooth function $f$ on $M$ the Laplacian of $F$ is defined by $\triangle (f)=div(df)=tr(Hes(f))$. Let $\{E_i\}_{i=1}^{n}$ be an orthonormal local basis and $\hat{i}=<E_i,E_i>=\pm1$. For computing the trace of $S_\theta$,we have 
\begin{align*}
tr(S_\theta)&=\sum^{n}_{i=1}\hat{i}S_\theta(E_i,E_i)=\sum^{n}_{i=1}\hat{i}(Hes(\theta)(E_i,E_i)+d\theta(E_i)d\theta(E_i))\\
=&tr(Hes(\theta))+\sum^{n}_{i=1}\hat{i}d\theta(E_i)d\theta(E_i)=\triangle(\theta)+\vert {\nabla}\theta\vert^2.
\end{align*}
\begin{proposition}
If $X,Y,Z\in\mathcal{X}(M)$ and $R$ is the curvature tensor of $M$, then the curvature tensor of $\hat{\nabla}$, denoted by $\hat{R}$, satisfies the following relations.  
\begin{align}
\hat{R}(\overline{X},G_1 )(G_1 )&=-[\overline{\nabla_X \nabla\theta}+d\theta (X)\overline{\nabla\theta}+\overline{F_\theta (F_\theta (X))}] \label{c}\\
\hat{R}(\overline{X},G_1 )(\overline{Y})&=[S_\theta (X,Y)-\langle F_\theta (X),F_\theta (Y) \rangle]G_1  \nonumber \\
&+\Omega_\theta (X,Y)\overline{\nabla\theta}-\overline{(\nabla_X F_\theta)(Y)}-d\theta(X)\overline{F_\theta(Y)}-d\theta(Y)\overline{F_\theta(X)}\\
\hat{R}(\overline{X},\overline{Y})(G_1 )&=\overline{(\nabla_Y F_\theta)(X)}-\overline{(\nabla_X F_\theta)(Y)}+2 \Omega_\theta (X,Y)\overline{\nabla\theta}\\
\hat{R}(\overline{X},\overline{Y})(\overline{Z})&=\overline{R(X,Y)(Z)}+\langle Z,F_\theta(X)\rangle \overline{F_\theta(Y)}-\langle Z,F_\theta(Y)\rangle \overline{F_\theta(X)}\nonumber \\
&+2\langle F_\theta(X),Y\rangle \overline{F_\theta(Z)}-2d\theta (Z)\Omega_\theta (X,Y) G_1 \nonumber\\
&+\langle (\nabla_X F_\theta)(Y)-(\nabla_Y F_\theta)(X) ,Z\rangle G_1
\end{align}
\end{proposition}
\begin{proof}
Routine computations show these results. For example we compute (\ref{c}),
\begin{align*}
\hat{R}(\overline{X},G_1 )(G_1 )&=\hat{\nabla}_{\overline{X}}\hat{\nabla}_{G_1 }G_1 -\hat{\nabla}_{G_1 }\hat{\nabla}_{\overline{X}}G_1 -\hat{\nabla}_{[\overline{X},G_1 ]}G_1 \\
&=-\hat{\nabla}_{\overline{X}}(\overline{\vec{\nabla}\theta}) + \hat{\nabla}_{G_1 }\overline{F_\theta(X)}+d\theta(X)\hat{\nabla}_{G_1 }G_1 \\
&=-\overline{\nabla_X \nabla\theta}-\Omega_\theta(X,\nabla\theta)G_1 -\overline{F_\theta(F_\theta (X))}+d\theta(F_\theta(X))G_1 \\
&-d\theta(X)\overline{\vec{\nabla }\theta}\\
&=-\overline{\nabla_X \nabla\theta}-d\theta(F_\theta (X)) G_1 -\overline{F_\theta (F_\theta (X))}+d\theta(F_\theta(X))G_1 )\\
&-d\theta(X)\overline{\vec{\nabla }\theta}\\
&=-[\overline{\nabla_X \nabla\theta}+d\theta (X)\overline{\nabla\theta}+\overline{F_\theta (F_\theta (X))}]
\end{align*}
\end{proof}

Note that in an inner vector space $V$ for each basis $\{ e_i\}_{i=1}^{k}$ of $V$, there is a basis $\{ e^i\}_{i=1}^{k}$ of $V$ which is known as its reciprocal basis, with the property $<e_i,e^j>=\delta^j_i$, where $\delta^j_i$ is the Kronecker delta symbol. Given any basis $\{ e_i\}_{i=1}^{k}$ of $V$ and its reciprocal basis, the trace of a linear map $T:V\longrightarrow V$ can be computed from the formula $tr(T)=\sum_{i=1}^{k}<T(e_i),e^i>$. For an alternative bilinear map $T:V\times V\longrightarrow W$, an straightforward computation shows that $\sum_{i=1}^{k}T(e_i,e^i)=0$. Inner product of linear map $T,S:V\longrightarrow W$, and bilinear map $T',S':V\times V\longrightarrow W$ (W has an inner product) is defined as follows.
\begin{align*}
<T,S>&=\sum_{i=1}^{k}<T(e_i),S(e^i)>\\
<T',S'>&=\sum_{i,j=1}^{k}<T'(e_i,e_j),S'(e^i,e^j)>
\end{align*}
To compute Ricci curvature and scalar curvature of $\hat{\nabla}$, we need to consider some local basis vector fields $\{E_i\}_{i=1}^{n}$ on $M$ with reciprocal basis $\{E^i\}_{i=1}^{n}$, in this case $\{\overline{E_1},\cdots,\overline{E_n},G_1 \}$ is a local basis for $\widehat{TM}$ and its reciprocal basis is $\{ \overline{E^1},\cdots ,\overline{E^n},G_1 \}$.
\\
In foregoing computations we use the following notation.
\begin{equation*}
Ric_\Omega^\theta (X,Y)=<F_\theta (X),F_\theta (Y))>=<\Omega_\theta (X,.),\Omega_\theta(Y,.)>.
\end{equation*}
\begin{proposition}
The Ricci curvature tensor of $\hat{\nabla}$, denoted by $\widehat{Ric}$, satisfies the following relations.
\begin{align}
\widehat{Ric}(G_1 ,G_1 ) &= -tr(F_\theta \circ F_\theta)-\triangle(\theta)-\vert \nabla \theta\vert^2 \label{d}\\
\widehat{Ric}(\overline{X},G_1) &=\langle div(F_\theta),X\rangle+2\langle F_\theta (\nabla\theta),X\rangle \\
\widehat{Ric}(\overline{X},\overline{Y} ) &= Ric(X,Y)-2Ric_\Omega^\theta (X,Y)-S_\theta (X,Y)
\end{align}
\end{proposition}
\begin{proof}
Routine computations show these results. For example we compute (\ref{d}),
\begin{align*}
\widehat{Ric}(G_1 ,G_1 )&=\langle \hat{R}(G_1 ,G_1 )(G_1 ),G_1 \rangle +\sum_{i=1}^n \langle \hat{R}(G_1 ,\overline{E_i})(\overline{E^i}),G_1 \rangle \\
&=\sum_{i=1}^n -S_\theta (E_i,E^i)+\langle F_\theta(E_i),F_\theta (E^i)\rangle  \\
&=-tr(S_\theta)-tr(F_\theta \circ F_\theta)
\end{align*}
\end{proof}
\begin{proposition}
Let $R$ be the scalar curvature of $M$ then the scalar curvature of $\hat{\nabla}$, denoted by $\hat{R}$ satisfies the following relation.
\begin{equation}
\hat{R}=R+tr(F_\theta\circ F_\theta)-2(\triangle(\theta)+\vert \nabla \theta\vert^2).
\end{equation}
\end{proposition}
\begin{proof}
\begin{align*}
\hat{R}&=\widehat{Ric}(G_1 ,G_1 )+\sum_{i=1}^n \widehat{Ric}(E_i , E^i)\\
&=\widehat{Ric}(G_1 ,G_1 )+\sum_{i=1}^n [Ric(E_i,E^i)-2Ric_\Omega^\theta (E_i,E^i)-S_\theta (E_i,E^i) \\
&=\widehat{Ric}(G_1 ,G_1 )+R+2tr(F_\theta \circ F_\theta)-tr(S_\theta)\\
&=R+tr(F_\theta \circ F_\theta)-2(\triangle(\theta)+\vert \nabla \theta\vert^2).
\end{align*}
\end{proof}
\section{Application to general relativity}
In this section, $M$ is an arbitrary oriented connected manifold which can be regarded as a space-time manifold.
\\
Fix a metric $\bar{g}=({\bf g},\theta)$ on $\widehat{TM}$, so for all $X,Y\in \mathcal{X}(M)$ and $f,g\in C^\infty (M)$ we have
\[
\bar{g}(\overline{X}+fG,\overline{Y}+gG )={\bf g}(X,Y)+fge^{2\theta}
\] 
Let $\bar{h}$ is an arbitrary metric on $\widehat{TM}$, we can write
\[
\bar{h}(\overline{X}+G,\overline{Y}+G )=\bar{h}(\overline{X},\overline{Y})+\bar{h}(\overline{X},G)+\bar{h}(G,\overline{Y})+\bar{h}(G,G)
\]
the above equation shows that $\bar{h}$ is completely determineds by its valus on $\bar{h}(\overline{X},\overline{Y})$, $\bar{h}(\overline{X},G)$ and $\bar{h}(G,G)$. The Riez representaion theorem shows there exist an 1- form $\delta$ on $M$ such that $\bar{h}(\overline{X},G)=e^{2\theta}\delta (X)$. So, $\bar{h}$ is determined by a triple $(g,\delta,h)$ where $g$ is a semmi-Riemannian metric on $M$, $\delta$ is an 1-form and, $h$ is a smooth function on $M$. In fact
\begin{align*}
\bar{h}(\overline{X},\overline{Y})&=g(X,Y)\\
\bar{h}(\overline{X},G)&=e^{2\theta}\delta (X)\\
\bar{h}(G,G)&=e^{2h}
\end{align*}
With above construction $\bar{g}$ is equivalent to the triple $(g,0,\theta)$ where, 0 is the zero 1-form on $M$. 
\\
Denote canonical volume form of a meter $g'$ on the oriented manifold $M$ by $dV_{g'}$. Also, denote the scalar curvature tensor of  metric $(g',\delta,\theta)$ on $\widehat{TM}$ by $\hat{R}'$. The Hilbert-Einstein action $\mathcal{L}$ on $\mathcal{M}$ is defined as follows.
\[
\mathcal{L}(g',\delta,\theta)=\int_M \hat{R}'\,dV_{g'}
\]
To be more precise, we must assume $M$ is compact or we must integrate on open subset $U$ of $M$ such that the closure $\overline{U}$ is compact.\\
For a symmetric two tensor $s$, a smooth function $h$ and 1- from $\delta$ set,
\begin{align*}
\tilde{g}(t)&=g+ts\\
\tilde{\delta}(t)&=t\delta\\
\tilde{\theta}(t)&=\theta +th
\end{align*}
For sufficiently small $t$, $\bar{\tilde{g}}=(\tilde{g}(t),\tilde{\delta}(t),\tilde{\theta}(t))$ is a variation of $\bar{g}=({\bf g},0,\theta)$. $\bar{g}$ is a critical meter for Hilber-Einstein action iff for any pair $(s,\delta,h)$:
\begin{equation}\label{lagrangian}
\dfrac{d}{dt}\vert_{t=0} \mathcal{L}\big(\tilde{g}(t),\tilde{\delta}(t),\tilde{\theta}(t)\big)=
\dfrac{d}{dt}\vert_{t=0} \int_M \hat{R}(t)dV_{{\bf g}+ts} =0
\end{equation}
where, $\hat{R}(t)$ is the scalar curvature of $\bar{\tilde{g}}=(\tilde{g}(t),\tilde{\delta}(t),\tilde{\theta}(t))$ and
\begin{equation*}
\hat{R}(t)=\tilde{R}(t)+tr(\tilde{F_\theta}\circ \tilde{F_\theta})-2(\triangle(\tilde{\theta}(t))+\vert \nabla \tilde{\theta}(t)\vert^2)
\end{equation*}
in above, $\tilde{R}(t)$ is the scalar curvature of $\tilde{g}(t)$. To find derivation in (\ref{lagrangian}), we must compute derivations of  $\tilde{R}(t)$ , $tr(\tilde{F_\theta}\circ \tilde{F_\theta})$  , $\triangle(\tilde{\theta}(t))$, $\vert \nabla \tilde{\theta}(t)\vert^2$  and $dV_{{\bf g}+ts}$ for $t=0$. In \cite{mass}, it is shown:
\begin{align}
\tilde{R}'(0)=&-<s,Ric>+div(X)\\
(dV_{{\bf g}+ts})'(0)=&\dfrac{1}{2}<{\bf g},s>dV_g\\
\triangle((\tilde{\theta})(t))'(0)=&\triangle(h)-div(s({\nabla} \theta))+\dfrac{1}{2}<{\nabla} tr(s),\nabla \theta>\\
(\vert \nabla \tilde{\theta}(t)\vert^2)'(0)=&-<s,d\theta\otimes d\theta>+2<\nabla h,\nabla \theta>
\end{align}
It is suficient to compute derivation of $tr(\tilde{F_\theta}\circ \tilde{F_\theta})$ at $t=0$. Tensor Field $F_\theta$ is obtained from $\Omega_\theta$ and, if $F_\theta =e^\theta F$ then,
\begin{equation*}
tr(FoF)=-\langle \Omega , \Omega \rangle , 
\end{equation*}
and 2-form $\Omega$ related to orthogonal decomposition of $\widehat{TM}$ with respect to $\bar{g}$. In fact, this decomposition defines a map $\lambda:TM\longrightarrow \widehat{TM}$, such that $\lambda (X) =\overline{X}$. And 
\begin{equation*}
2\Omega(X,Y) = [\overline{X},\overline{Y}]-\overline{[X,Y]}=[\lambda(X),\lambda(Y)]-\lambda([X,Y])
\end{equation*}
Since for all $X\in \mathcal{X}(M)$, we have 
\begin{equation*}
\bar{\tilde{g}}(\lambda(X)+\Gamma_t(X)G,G)=0,
\end{equation*}
where,
\begin{equation*}
\Gamma_t (X)=-te^{-2th}\delta(X)
\end{equation*}
So, $X\mapsto \lambda(X)+\Gamma_t (X)G$ is the induced map by the metric $(\tilde{g},\tilde{\delta},\tilde{\theta})$. Hence, we can write
\begin{align*}
2\tilde{\Omega}(t)(X,Y)G&=[\lambda(X)+\Gamma_t (X)G,\lambda(Y)+\Gamma_t (Y)G]-(\lambda[X,Y]+\Gamma_t ([X,Y])G)\\
&=[\lambda(X),\lambda(Y)]-\lambda([X,Y])+X(\Gamma_t (Y)G)-Y(\Gamma_t (X)G)-\Gamma_t ([X,Y])G\\
&=[\lambda(X),\lambda(Y)]-\lambda([X,Y])+d\Gamma_t(X,Y)G
\end{align*}
This yields,
\[
2\tilde{\Omega}'(0)= -d\delta
\]
By means of a local coordinate system on $M$ with local frame $\partial_i$, set $\Omega(\partial_i,\partial_j)=\Omega_{ij}$. we have
\begin{align*}
<\Omega,\Omega>&= g^{im}g^{jn}<\Omega_{ij},\Omega_{mn}>
\end{align*}
consequently,
\begin{align*}
\big(\langle \tilde{\Omega}(t),\tilde{\Omega}(t)\rangle \big)'(0)=&\big({\tilde{g}(t)}^{im}{\tilde{g}(t)}^{jn}\langle {\tilde{\Omega}(t)}_{ij},{\tilde{\Omega}(t)}_{mn}\rangle \big)'(0)\\
=&-s^{im}g^{jn}\langle \Omega_{ij},\Omega_{mn}\rangle 
-g^{im}s^{jn}\langle \Omega_{ij},\Omega_{mn}\rangle \\
-&\dfrac{1}{2}g^{im}g^{jn}\langle d\delta_{ij},\Omega_{mn}\rangle 
-\dfrac{1}{2}g^{im}g^{jn}\langle \Omega_{ij},d\delta_{mn}\rangle \\
=&-2s^{im}g^{jn}\langle \Omega_{ij},\Omega_{mn}\rangle 
-g^{im}g^{jn}\langle d\delta_{ij},\Omega_{mn}\rangle 
\end{align*}
so, the above computations show
\begin{equation*}
\big(\langle \tilde{\Omega}(t),\tilde{\Omega}(t)\rangle \big)'(0)=-2\langle Ric_\Omega,s\rangle -\langle d\delta,\Omega\rangle 
\end{equation*}
where $Ric_\Omega =e^{-2\theta}Ric_\Omega^\theta$. hence,
\begin{equation*}
\dfrac{d}{dt}\vert_{t=0} tr(\tilde{F}\circ \tilde{F}) =2\langle Ric_\Omega,s\rangle +\langle d\delta,\Omega\rangle 
\end{equation*}
so,
\begin{equation}
\dfrac{d}{dt}\vert_{t=0} tr(\tilde{F_\theta}\circ \tilde{F_\theta}) =2\langle Ric_\Omega^\theta ,s\rangle +e^{\theta}\langle d\delta,\Omega_\theta\rangle +2h.tr(F_\theta \circ F_\theta)
\end{equation}
For every $T\in A^p(M)$ and $S\in A^{p-1}(M,)$, we have \cite{bleecker}
\begin{equation*}
\int_{M}<d S,T>dV_g=\int_{M}<S,-div(T)>dV_g
\end{equation*}
Also note that the integral of divergence of every vector fields on $M$ is zero. Now, we are ready to compute derivation of the $\mathcal{L}\big(\tilde{g}(t),\tilde{\delta}(t),\tilde{\theta}(t)\big)$ for $t=0$. We can write
\begin{align*}
&\dfrac{d}{dt}\vert_{t=0} \mathcal{L}\big(\tilde{g}(t),\tilde{A}(t),\tilde{\theta}(t)\big)=
\dfrac{d}{dt}\vert_{t=0} \int_M \hat{R}(t)dV_{g+ts} =\int_M \big(\hat{R}(t)dV_{g+ts}\big)'(0)\\
&=\int_M \big[\big(\tilde{R}(t)+tr(\tilde{F_\theta}\circ \tilde{F_\theta})-2(\triangle(\tilde{\theta}(t))+\vert \nabla \tilde{\theta}(t)\vert^2)\big)dV_{g+ts}\big]'(0)\\
&=\int_M \big[\big(\tilde{R}(t)-2(\triangle(\tilde{\theta}(t))+\vert \nabla \tilde{\theta}(t)\vert^2)\big)dV_{g+ts}\big]'(0)\\
&+\int_M \big[\big(tr(\tilde{F_\theta}\circ \tilde{F_\theta})\big)dV_{g+ts}\big]'(0)\\
&=I+I\!I
\end{align*}
Where $I$ is the first integral and, $I\!I$ is the second one. According to \cite{mass}, we have
\[
I=\int_M \big[\langle s,-Ric+(\dfrac{1}{2}R-\vert \vec{\nabla}\theta \vert^2){\bf g}+2d\theta\otimes d\theta\rangle+4\triangle(\theta)h\big]dV_{\bf g}
\]
We compute $I\!I$ as follows,
\begin{align*}
I\!I &=\int_M \big[\big(tr(\tilde{F_\theta}\circ \tilde{F_\theta})\big)dV_{g+ts}\big]'(0)\\
&=\int_M tr(F_\theta \circ F_\theta)\dfrac{1}{2}\langle s,g\rangle dV_g\\
&+\int_M \big[ 2\langle Ric_\Omega^\theta ,s\rangle +e^{\theta}\langle d\delta,\Omega_\theta\rangle +2h.tr(F_\theta \circ F_\theta)\big]dV_g\\
&=\int_M\big[ \langle 2Ric_\Omega^\theta+\dfrac{1}{2}tr(F_\theta \circ F_\theta)g,s\rangle -\langle \delta, div(e^{\theta}\Omega_\theta)\rangle \rangle +2h.tr(F_\theta \circ F_\theta)\big]dV_g\\
\end{align*}
consequently,
\begin{align*}
I+I\!I &=\int_M \big[\langle s,-Ric+2Ric_\Omega^\theta +(\dfrac{1}{2} R +\dfrac{1}{2}tr(F_\theta \circ F_\theta)-\vert \nabla\theta\vert^2)g +2d\theta\otimes d\theta\rangle\\
&-\langle \delta , div(e^\theta\Omega_\theta) \rangle +(2tr(F_\theta \circ F_\theta)+4\triangle(\theta)).h \big]dV_g
\end{align*}
The above expresion vanishes for all pair $(s,\delta,h)$, iff
\begin{align}
Ric-\dfrac{1}{2}Rg&=2Ric_\Omega^\theta+2d\theta\otimes d\theta \nonumber \\
&+(\dfrac{1}{2}tr(F_\theta \circ F_\theta)-\vert \nabla\theta \vert^2)g\label{m1}\\
div(e^\theta\Omega_\theta )&=0\label{m2}\\
\triangle(\theta)&=-\dfrac{1}{2}tr(F_\theta \circ F_\theta) \label{m3}
\end{align}
These are field equations for $(g,0,\theta)$ and determine critical metrics for Hilbert action. The first equation is Einstein field equation. Taking traces of both sides of (\ref{m1}), we find that
\[
R=2\vert \nabla\theta\vert^2
\]
the value of $R$ is related to the existent matter in the space-time points and, the above relation shows $\theta$ is related to matter. According to this equality we can write the field equations as follows.
\begin{align}
Ric &=2Ric_\Omega^\theta+2d\theta\otimes d\theta 
+\dfrac{1}{2}tr(F_\theta \circ F_\theta)g \label{M1}\\
div(e^\theta \Omega_\theta )&=0\label{M2}\\
\triangle(\theta)&=-\dfrac{1}{2}tr(F_\theta \circ F_\theta)\label{M3}
\end{align}
Note that for a tensor field $T$, and an smooth function $f$ on $M$, $div(fT)=fdiv(T)+i_{\nabla f}T$. 
According to the equation (\ref{M2}), for all $X\in\mathcal{X}(M)$ we can write
\begin{align*}
0=div(e^\theta \Omega_\theta)(X)&=div(e^{2\theta}\Omega)(X)\\
&=e^{2\theta}div(\Omega)(X)+2e^{2\theta}\Omega (\nabla\theta , X)
\end{align*}
Since $div(\Omega)(X)=\langle div(F), X\rangle$, so (\ref{M2}) is equivalent to 
\begin{equation}
\langle div(F), X\rangle =2d\theta(F(X))\label{div}
\end{equation}
In seuel we use the following notations.
\begin{align*}
T^\Omega_\theta (X,Y)&=2Ric_\Omega^\theta +\dfrac{1}{2}tr(F_\theta \circ F_\theta)g(X,Y)\\
T^\theta &=2d\theta \otimes d\theta -\vert \nabla\theta\vert^2 g
\end{align*}
In the term of this notation (\ref{m1}) can be written as follows.
\begin{equation}
Ric -\dfrac{1}{2}Rg=T^\Omega_\theta +T^\theta
\end{equation}
We compute divergence of both tensors $Ric_\Omega$ and $tr(F\circ F) g$. To compute $(div(Ric_\Omega))(X)$, fix a point $p$ in $M$. Assume that $(\nabla X)(p)=0$, by parallel translating the value of $X$ at $p$. Also, take a normal basis $\{E_i\}_{i=1}^{n}$ at $p$. The reciprocal basis $\{E^i\}_{i=1}^{n}$ is also normal and, we compute at $p$.
\begin{align*}
div(Ric_\Omega)(X)&=\sum_{i=1}^{n}(\nabla_{E_i} Ric_\Omega)(E^i,X)=\sum_{i=1}^{n}{E_i}( Ric_\Omega(E^i,X))\\
&=\sum_{i=1}^{n} E_i(\langle F(E^i),F(X)\rangle ) \\
&=\sum_{i=1}^{n} \langle (\nabla_{E_i} F)(E^i),F(X)\rangle +\sum_{i=1}^{n}\langle F(E^i),(\nabla_{E_i} F)(X)\rangle \\
&=\langle div(F),F(X)\rangle +\sum_{i=1}^{n}\langle F(E^i),(\nabla_{E_i} F)(X)\rangle \\
&=\langle div(F),F(X)\rangle +\sum_{i,j=1}^{n}\langle F(E^i),E_j\rangle \langle (\nabla_{E_i} F)(X),E^j\rangle \\
&=\langle div(F),F(X)\rangle +\dfrac{1}{2} \sum_{i,j=1}^{n}\langle F(E^i),E_j\rangle \big[ \langle (\nabla_{E_i} F)(X),E^j\rangle -\langle (\nabla_{E^j} F)(X),E_i\rangle  \big]\\
&=\langle div(F),F(X)\rangle +\dfrac{1}{2} \sum_{i,j=1}^{n}\langle F(E^i),E_j\rangle \big[ \langle (\nabla_{E_i} F)(X),E^j\rangle +\langle (\nabla_{E^j} F)(E_i),X\rangle  \big]\\
&=\langle div(F),F(X)\rangle +\dfrac{1}{2} \sum_{i,j=1}^{n}\langle F(E^i),E_j\rangle (-\langle (\nabla_X F)(E^j),E_i\rangle )\\
&=\langle div(F),F(X)\rangle +\dfrac{1}{2} \sum_{i,j=1}^{n}\langle F(E^i),E_j\rangle \langle (\nabla_X F)(E_i),E^j\rangle \\
&=\langle div(F),F(X)\rangle +\dfrac{1}{2} \sum_{i=1}^{n}\langle F(E^i),(\nabla_X F)(E_i)\rangle \\
&=\langle div(F),F(X)\rangle +\dfrac{1}{4}X(\vert F\vert^2)
\end{align*}
In above we used the differential identity for $F$ which is derived from (\ref{bianchi}).
\begin{equation*}
\langle (\nabla_X F)(Y),Z\rangle +\langle (\nabla_Y F)(Z),X\rangle +\langle (\nabla_Z F)(X),Y\rangle =0
\end{equation*}
Remind that for any smooth function $f$ and the metric $\textbf{g}$ on $M$: \\$(div(f\textbf{g}))(X)=X(f)=df(X)$. So,
\begin{align*}
div(tr(F\circ F) \mathit{g})&=X(tr(F\circ F))=X(\sum_{i=1}^{n}\langle (F\circ F)(E^i),E_i\rangle ) \\
&= - X(\sum_{i=1}^{n}\langle  F(E^i),F(E_i)\rangle )=-X(\vert F\vert^2)
\end{align*}
hence, with the help of (\ref{div}) we compute divergence of $T^\Omega=e^{-2\theta}T^\Omega_\theta$,
\begin{align*}
div(T^\Omega)(X)&=2div(Ric_\Omega)(X)+\dfrac{1}{2}div(tr(F\circ F) g)(X)\\
&=2\langle div(F),F(X)\rangle +\dfrac{1}{2}X(\vert F \vert^2)-\dfrac{1}{2}X(\vert F \vert^2)\\
&=4d\theta (F(F(X))
\end{align*}
Also, $div(T^\Omega_\theta)=div(e^{2\theta}T^\Omega)$, So
\begin{align*}
div(T^\Omega_\theta)(X)&=e^{2\theta}div(T^\Omega)(X)+2e^{2\theta}T^\Omega (\nabla\theta , X)\\
&=4e^{2\theta}d\theta(F(F(X))+2e^{2\theta}\big( 2\langle F(\nabla\theta),F(X)\rangle +\dfrac{1}{2}tr(F\circ F)g(\nabla\theta , X)\big)\\
&=4e^{2\theta}d\theta(F(F(X))-4e^{2\theta}d\theta(F(F(X))+tr(F_\theta \circ F_\theta)d\theta (X)\\
&=tr(F_\theta \circ F_\theta)d\theta(X)
\end{align*}
In summary, with the help of (\ref{M2}) which is equivalent to (\ref{div}), we have
\begin{equation}\label{Div}
div(T^\Omega_\theta)(X)=tr(F_\theta \circ F_\theta)d\theta(X)
\end{equation}
To compute divergence of the symmetric tensor $T^\theta$ we can write.
\begin{align*}
div(d\theta\otimes d\theta)(X)&=\sum_{i=1}^n (\nabla_{E_i} d\theta\otimes d\theta)(E^i ,X)\\
&=\sum_{i=1}^n \big((\nabla_{E_i} d\theta)\otimes d\theta+ d\theta\otimes (\nabla_{E_i} d\theta)\big)(E^i ,X)\\
&=(\nabla_{E_i} d\theta)(E^i)d\theta(X)+d\theta(E^i)(\nabla_{E_i} d\theta)(X)\\
&=\triangle(\theta)d\theta(X)+\sum_{i=1}^n d\theta(E^i)Hes(\theta)(E_i , X)=\triangle(\theta)d\theta (X)+Hes(\vec{\nabla}\theta ,X)
\end{align*}
Also, we have
\begin{align*}
div(\vert \vec{\nabla}\theta\vert^2 g)(X)&=d(\vert \vec{\nabla}\theta\vert^2)(X)=X\langle \vec{\nabla}\theta ,\vec{\nabla}\theta \rangle =2\langle  \nabla_X (\vec{\nabla}\theta),\vec{\nabla}\theta \rangle \\
&=2(\nabla_X d\theta)(\vec{\nabla}\theta)=2Hes(\theta)(\vec{\nabla}\theta ,X).
\end{align*}
hence
\[
div(T^\theta)=2\triangle (\theta)d\theta (X)
\]
Since, divergence of Einstein tensor is zero so, the divergence of the right hand side of (\ref{m1}) must be zero. With the help of (\ref{M2}) and (\ref{M3}), we have the following proposition that remindes the conversation low for momentum-energy tensor.
\begin{proposition}
The divergence of symmetric tensor $T^\Omega_\theta +T^\theta$ is zero.
\end{proposition}
\begin{proof}
It is sufficient to collect our computations in above.
\begin{align*}
div(T^\Omega_\theta +T^\Omega)(X)&=div(T^\Omega_\theta)(X) +div(T^\theta )(X)\\
&=tr(F_\theta \circ F_\theta)d\theta(X)+2\triangle (\theta)d\theta (X)\\
&=tr(F_\theta \circ F_\theta)d\theta(X)-tr(F_\theta \circ F_\theta)d\theta(X)=0
\end{align*}
\end{proof}
\section{Conclusion}
This structure is completely geometrical and forces and matter are parts of the geometry.\\
Of course, it does not contain quantum effects. This theory must be improved such that be capable of describing particles and their internal structures.


\begin{thebibliography}{9}
\bibitem{bleecker} D.Bleecker, "Guage Theory and Variational Principles", Addison-wesely, 1981.
\bibitem{mass}N.Boroojerdian, "Geometrization of Mass in General Relativity", Int.J.Theor.Phys; Jul2013, Vol. 52 Issue 7, p2432.
\bibitem{buc} M.Boucetta, "Riemannian Geometry of the Lie Algebroids", Journal of the Egyptian Mathematical Society,
Volume 19, Issues 1–2, April–July 2011, Pages 57-70.
\bibitem{affine 1} N. Elyasi, N. Boroojerdian, "Affine Metrics: An Structure for Unification of Gravitation and Electromagnetism", Int.J.Theor.Phys; March 2011, Volume 50, Issue 3, pp 850-860.
\bibitem{affine 2}N. Elyasi, N. Boroojerdian, "Affine metrics and algebroid structures: Application to general relativity and unification",
Int.J.Theor.Phys;October 2012, Volume 51, Issue 10, pp 3160-3167.
\bibitem{me}Gh.Fasihi Ramandi, N.Boroojerdian, "Forces Unification in The Framework of Transitive Lie Algebroids", Int.J.Theor.Phys; October2014. DOI: 10.1007/s10773-014-2357-5.
\bibitem{history}Goenner, H.F.M.: On the History of Unified Field Theories. Max Planck Institute for Gravitational
Physics, Albert Einstein Institute (2004).
\end{thebibliography}
\end{document}